\documentclass[a4paper]{amsart}
\usepackage{amsmath}
\usepackage{amsthm}
\usepackage{amsfonts}
\usepackage{amssymb}
\usepackage{cite}
\usepackage{enumerate}
\usepackage{mathrsfs}
\usepackage{eucal}

\newtheorem{theorem}{Theorem}[section]

\newtheorem{lemma}[theorem]{Lemma}
\newtheorem{proposition}[theorem]{Proposition}

\theoremstyle{definition}

\newtheorem{question}[theorem]{Question}
\newtheorem{remark}[theorem]{Remark}

\newcommand{\fu}{finitely universal} \newcommand{\Fu}{Finitely universal}
 \newcommand{\Lu}{Lattice universal}
\newcommand{\Ocom}{Overcommutative} \newcommand{\ocom}{overcommutative}

\newcommand{\Id}{\mathsf{Id}}

\newcommand{\ida}{\normalfont\texttt{A}}
\newcommand{\idb}{\normalfont\texttt{B}}
\newcommand{\idc}{\normalfont\texttt{C}}
\newcommand{\idd}{\normalfont\texttt{D}}
\newcommand{\ide}{\normalfont\texttt{E}}
\newcommand{\idO}{0}

\newcommand{\lattice}{\mathfrak}
\newcommand{\lL}{\lattice{L}}
\newcommand{\lEq}{\lattice{Eq}}

\newcommand{\monoid}{ }
\newcommand{\mA}{\monoid{A}}
\newcommand{\mB}{\monoid{B}}
\newcommand{\mM}{\monoid{M}}
\newcommand{\mN}{\monoid{N}}
\newcommand{\mS}{\monoid{S}}
\newcommand{\mT}{\monoid{T}}

\newcommand{\word}{\mathbf}
\newcommand{\wa}{\word{a}}
\newcommand{\wb}{\word{b}}
\newcommand{\ws}{\word{s}}
\newcommand{\wt}{\word{t}}
\newcommand{\wu}{\word{u}}
\newcommand{\wv}{\word{v}}
\newcommand{\ww}{\word{w}}

\newcommand{\setofwords}{\mathscr}
\newcommand{\sB}{\setofwords{B}}
\newcommand{\sD}{\setofwords{D}}
\newcommand{\sW}{\setofwords{W}}
\newcommand{\sX}{\setofwords{X}}

\newcommand{\svariety}{\mathbf}
\newcommand{\svCom}{\svariety{C}\text{\scriptsize$\svariety{OM}$}}
\newcommand{\svH}{\svariety{H}}
\newcommand{\svV}{\svariety{V}}

\newcommand{\variety}{\mathbb}
\newcommand{\vA}{\variety{A}}
\newcommand{\vB}{\variety{B}}
\newcommand{\vC}{\variety{C}}
\newcommand{\vCom}{\variety{C}\text{\scriptsize$\variety{OM}$}}
\newcommand{\vE}{\variety{E}}
\newcommand{\vJ}{\variety{J}}
\newcommand{\vM}{\variety{M}}
\newcommand{\vMon}{\variety{M}\text{\scriptsize$\variety{ON}$}}
\newcommand{\vO}{\variety{O}}

\newcommand{\vT}{\variety{T}}
\newcommand{\vV}{\variety{V}}

\newcommand{\Tuma}{T$\mathring{\text{u}}$ma}

\allowdisplaybreaks

\begin{document}
\title[Varieties of monoids with complex lattices of subvarieties]{Varieties of monoids with complex lattices of subvarieties}

\begin{abstract}
A variety is \textit{finitely universal} if its lattice of subvarieties contains an isomorphic copy of every finite lattice.
Examples of finitely universal varieties of semigroups have been available since the early 1970s, but it is unknown if there exists a finitely universal variety of monoids.
The main objective of the present article is to exhibit the first examples of finitely universal varieties of monoids.
The finite universality of these varieties is established by showing that the lattice of equivalence relations on every sufficiently large finite set is anti-isomorphic to some subinterval of the lattice of subvarieties.
\end{abstract}

\author[S.\@ V.\@ Gusev]{Sergey V.\@ Gusev}
\address{Institute of Natural Sciences and Mathematics, Ural Federal University, Lenina str.\@~51, 620000 Ekaterinburg, Russia}
\email{sergey.gusb@gmail.com}

\author[E.\@ W.\@ H.\@ Lee]{Edmond W.\@ H.\@ Lee}
\address{Department of Mathematics, Nova Southeastern University, Fort Lauderdale, FL 33314, USA}
\email{edmond.lee@nova.edu}

\thanks{The first author was supported by the Ministry of Science and Higher Education of the Russian Federation (project FEUZ-2020-0016).}

\maketitle

\section{Introduction}

\subsection{{\Fu} varieties} \label{subsec: Fu}

A \textit{variety} is a class of algebras of a fixed type that is closed under the formation of homomorphic images, subalgebras, and arbitrary direct products.
Following Shevrin et al.\@~\cite{SVV09}, a variety~$\svV$ is \textit{\fu} if its lattice $\lL(\svV)$ of subvarieties contains an isomorphic copy of every finite lattice.
Examples of {\fu} varieties include the variety $\svCom$ of all commutative semigroups and the variety~$\svH$ of semigroups defined by the identity \[ x^2 \approx yxy; \] see Burris and Nelson~\cite{BN71b} and Volkov~\cite{Vol89}, respectively.
Not only is the variety $\svCom$ non-finitely generated, it is not contained in any finitely generated variety.
The variety~$\svH$ is also non-finitely generated~\cite{Lee10}, but it is a more interesting example because it is contained in every known example of {\fu} variety generated by a finite semigroup, such as the well-known Brandt semigroup \[ \mB_2 = \langle a,b \,|\, a^2=b^2=0, \, aba=a, \, bab=b \rangle \] of order five and even several semigroups of order four~\cite{Lee07}.

Recall that a \textit{monoid} is a semigroup with an identity element.
Every group is a monoid, but a semigroup need not be a monoid in general.
For a semigroup~$\mS$ that is not a monoid, an external identity element~$1$ can be adjoined to it to obtain a monoid, which is denoted by~$\mS^1$.
Given a monoid~$\mM$ with identity element~$1$, it is possible to consider~$\mM$ as either an algebra with just an associative binary operation (semigroup signature) or an algebra with both an associative binary operation and a nullary operation fixing~$1$ (monoid signature).
Some results remain unchanged regardless of which signature type is considered, for example, a finite monoid is finitely based as a semigroup if and only if it is finitely based as a monoid \cite[Section~1]{Vol01}.
However, results can be drastically different at the varietal level, as demonstrated by the monoid~$\mN_6^1$ obtained from the nilpotent semigroup \[ \mN_6 = \langle a,b \,|\, a^2=b^2=bab=0\rangle \] of order six: the variety of semigroups generated by~$\mN_6^1$ has uncountably many subvarieties~\cite{Jac00}, while the variety of monoids generated by~$\mN_6^1$ has only five subvarieties~\cite{Jac05}.
The main reason for this huge ``loss" of varieties when switching from the semigroup signature to the monoid signature is due to the possibility in the latter to remove any variables from identities by substituting~$1$ into them.
The aforementioned identity $x^2 \approx yxy$, which defines the {\fu} variety~$\svH$ of semigroups, clearly illustrates this phenomenon: any variety of monoids that satisfies this identity also satisfies the identities $\{ x^2 \approx x,\, 1 \approx y^2 \}$ and so is trivial.

Nevertheless, there exist finite monoids that generate varieties of monoids with uncountably many subvarieties, for example, the monoid~$\mB_2^1$ of order six and the monoid~$\mN_8^1$ obtained from the nilpotent semigroup \[ \mN_8 = \langle a,b \,|\, a^2=b^2=baba=0\rangle \] of order eight~\cite{JL18}.
However, it is unknown if there exists a {\fu} variety of monoids.
This led to the following question posed in a recent study of varieties of monoids with extreme properties.

\begin{question}[Jackson and Lee~{\cite[Question~6.3]{JL18}}] \label{Q: Mon FU}
Is the variety~$\vMon$ of all monoids {\fu}?
\end{question}

A variety is \textit{periodic} if it satisfies the identity $x^{m+k} \approx x^m$ for some $m,k \geq 1$; a variety is \textit{aperiodic} if it satisfies such an identity with $k = 1$.
It is well known that a variety that is not periodic contains the variety~$\vCom$ of all commutative monoids and so is said to be \textit{\ocom}.
{\Ocom} varieties constitute the interval $[\vCom,\vMon]$.
Recently, Gusev~\cite{Gus18} proved that the interval $[\vCom,\vMon]$ violates every nontrivial lattice identity; the complexness of this interval naturally led to the following question.

\begin{question}[Gusev~{\cite[Question]{Gus18}}] \label{Q: OC FU}
Does the interval $[\vCom,\vMon]$ contain an isomorphic copy of every finite lattice?
\end{question}

An affirmative answer to Question~\ref{Q: OC FU} clearly implies an affirmative answer to Question~\ref{Q: Mon FU}.

\subsection{Main results and organization} \label{subsec: main results}

Unless otherwise specified, all varieties in this article are varieties of monoids.
The variety~$\vO$ defined by the following identities is central to the present investigation: \[ xyt_1xt_2y \approx yxt_1xt_2y, \quad xt_1xyt_2y \approx xt_1yxt_2y, \quad xt_1yt_2xy \approx xt_1yt_2yx. \tag{$\idO$} \label{id: O} \]
It is clear that every commutative monoid satisfies the identities~\eqref{id: O}, so that~$\vO$ is {\ocom}.
The identities~\eqref{id: O} played a prominent role in the construction of many important varieties \cite{Gus19,Jac05,Lee13a,Lee14,Lee15} and are crucial to the study of the finite basis problem \cite{Lee12a,Sap16,Sap19}.
Subvarieties of~$\vO$ are defined by identities of a very specific form (see Lemma~\ref{L: O subvarieties}), and this is favorable to the search for appropriate sublattices of $\lL(\vO)$ to address Questions~\ref{Q: Mon FU} and~\ref{Q: OC FU}.

Some background results are first given in Section~\ref{sec: prelim}.
It is then shown in Section~\ref{sec: overcom} that for each $n \geq 3$, the lattice of equivalence relations on an $n$-element set is anti-isomorphic to some subinterval of $[\vCom,\vO]$.
It follows from the well-known theorem of Pudl\'ak and \Tuma~\cite{PT80} that every finite lattice is embeddable in $[\vCom,\vO]$.
This affirmatively answers Question~\ref{Q: OC FU} and so also Question~\ref{Q: Mon FU}.

\begin{remark}
Since the class of all finite lattices violates every nontrivial quasi-iden\-ti\-ty~\cite{BG75}, the interval $[\vCom,\vO]$ also violates every nontrivial quasi-identity.
This generalizes the aforementioned result of Gusev~\cite{Gus18} about the interval $[\vCom,\vMon]$.
More generally, for any {\fu} variety~$\vV$, the lattice $\lL(\vV)$ violates every nontrivial quasi-identity.
\end{remark}

Given that {\ocom} varieties constitute the top region $[\vCom,\vMon]$ of the lattice $\lL(\vMon)$, an obvious next step in the investigation is to consider if there exists a periodic variety that is {\fu}.
To this end, the subvariety~$\vE_m$ of~$\vO$ defined by the identities $\{ x^{m+1} \approx x^m, \, x^mt \approx tx^m \}$ plays a useful role.
Evidently, the variety~$\vE_m$ is aperiodic and the proper inclusions $\vE_1 \subset \vE_2 \subset \vE_3 \subset \cdots$ hold.
Since all subvarieties of~$\vO$ are finitely based~\cite{Lee12a}, the variety~$\vE_m$ contains at most countably many subvarieties.

The variety~$\vE_1$ coincides with the variety of semilattice monoids: it is finitely generated and has only two subvarieties.
For each $m \geq 2$, the variety~$\vE_m$ is non-finitely generated \cite{Lee15}.
The lattice $\lL(\vE_2)$ is a countably infinite chain~\cite{Lee15}, but not much is known about the structure of the lattice $\lL(\vE_m)$ when $m \geq 3$.
In Section~\ref{sec: periodic}, the variety~$\vE_3$ is shown to be {\fu}; this result is obtained by showing that for each $n \geq 2$, the lattice of equivalence relations on an $n$-element set is anti-isomorphic to some subinterval of $\lL(\vE_3)$.
Consequently, the aperiodic variety~$\vE_m$ is {\fu} for all $m \geq 3$.
\begin{table}[h] \caption{Some properties satisfied by $\vE_m$}
\begin{tabular}{rccc} \hline
                                  & $m = 1$ & $m=2$ & $m\geq3$ \\ \hline
Number of subvarieties & $2$ & $\aleph_0$ & $\aleph_0$ \\
Finitely generated     & Yes & No         & No \\
{\Fu}                  & No  & No         & Yes \\ \hline
\end{tabular}
\end{table}

As noted earlier, there exist {\fu} varieties of semigroups that are finitely generated.
In contrast, by the celebrated theorem of Oates and Powell~\cite{OP64}, every finitely generated variety of groups contains only finitely many subvarieties and so is not {\fu}.
Given that the class of finite monoids is properly sandwiched between the class of finite semigroups and the class of finite groups, it is of fundamental interest to question the existence of a {\fu} variety of monoids that is finitely generated.
Although the {\fu} varieties $\vE_3,\vE_4,\vE_5,\ldots$ are non-finitely generated~\cite{Lee15}, they turn out to be relevant.
In Section~\ref{sec: FG}, each variety~$\vE_m$ is shown to be contained in some finitely generated variety.
Consequently, there exist {\fu} varieties of monoids that are finitely generated.

The article ends with some open problems in Section~\ref{sec: problems}.

\section{Preliminaries} \label{sec: prelim}

Acquaintance with rudiments of universal algebra is assumed of the reader.
Refer to the monograph of Burris and Sankappanavar~\cite{BS81} for more information.

\subsection{Words, identities, and deduction}

Let~$\sX^*$ denote the free monoid over a countably infinite alphabet~$\sX$.
Elements of~$\sX$ are called \textit{variables} and elements of~$\sX^*$ are called \textit{words}.

An identity is written as $\wu \approx \wv$, where $\wu,\wv \in \sX^*$; it is \textit{nontrivial} if $\wu \neq \wv$.
An identity $\wu \approx \wv$ is \textit{directly deducible} from an identity $\ws \approx \wt$ if there exist some words $\wa,\wb \in \sX^*$ and substitution $\varphi: \sX \to \sX^*$ such that $\{ \wu,\wv \} = \big\{ \wa\varphi(\ws)\wb,\wa\varphi(\wt)\wb \big\}$.
A nontrivial identity $\wu \approx \wv$ is \textit{deducible} from a set~$\Sigma$ of identities, indicated by $\Sigma \vdash \wu \approx \wv$, if there exists some finite sequence $\wu = \ww_0, \ww_1, \ldots, \ww_m = \wv$ of distinct words such that each identity $\ww_i \approx \ww_{i+1}$ is directly deducible from some identity in~$\Sigma$.
Informally, the deduction $\Sigma \vdash \wu \approx \wv$ means that the identities in~$\Sigma$ can be used to \textit{convert}~$\wu$ into~$\wv$.

\begin{theorem}[Birkhoff's Completeness Theorem for Equational Logic; see Burris and Sankappanavar {\cite[Theorem~II.14.19]{BS81}}] \label{T: deduction}
Let~$\vV$ be the variety defined by some set~$\Sigma$ of identities.
Then~$\vV$ satisfies an identity $\wu \approx \wv$ if and only if $\Sigma \vdash \wu \approx \wv$.
\end{theorem}

Two sets of identities~$\Sigma_1$ and~$\Sigma_2$ are \textit{equivalent}, indicated by $\Sigma_1 \sim \Sigma_2$, if the deductions $\Sigma_1 \vdash \Sigma_2$ and $\Sigma_2 \vdash \Sigma_1$ hold.
In other words, $\Sigma_1 \sim \Sigma_2$ if and only if~$\Sigma_1$ and~$\Sigma_2$ define the same variety.
The subvariety of a variety~$\vV$ defined by a set~$\Sigma$ of identities is denoted by~$\vV\Sigma$.

For any set~$\mathsf{A}$, the lattice of equivalence relations on~$\mathsf{A}$ is denoted by $\lEq(\mathsf{A})$ and the equality relation on~$\mathsf{A}$ is denoted by~$\varepsilon_\mathsf{A}$.
Given any set $\sW \subseteq \sX^*$ of words and any equivalence relation $\pi \in \lEq(\sW)$, define \[ \Id(\pi) = \{ \wu \approx \wv \,|\, (\wu,\wv) \in \pi \}. \]

\begin{lemma} \label{L: equivalence}
For any $\sW \subseteq \sX^*$ and $\Sigma \subseteq \{ \wu \approx \wv \,|\, \wu,\wv \in \sW \}$, there exists some $\pi \in \lEq(\sW)$ such that $\Sigma \sim \Id(\pi)$.
\end{lemma}

\begin{proof}
Let~$\widehat{\Sigma}$ be the smallest set of identities containing~$\Sigma$ such that
\begin{enumerate}[(1)]
\item[$\bullet$] $\wu \approx \wu \in \widehat{\Sigma}$ for all $\wu \in \sW$;
\item[$\bullet$] if $\wu \approx \wv \in \widehat{\Sigma}$, then $\wv \approx \wu \in \widehat{\Sigma}$; and
\item[$\bullet$] if $\wu \approx \wv, \, \wv \approx \ww \in \widehat{\Sigma}$, then $\wu \approx \ww \in \widehat{\Sigma}$.
\end{enumerate}
Then $\Sigma \sim \widehat{\Sigma}$ and $\widehat{\Sigma} = \Id(\pi)$ with $\pi = \{ (\wu,\wv) \,|\, \wu \approx \wv \in \widehat{\Sigma} \} \in \lEq(\sW)$.
\end{proof}

\subsection{Rigid words and identities} \label{subsec: rigid}

Define a \textit{rigid word} to be the word \[ x^{e_0} \prod_{i=1}^r (t_ix^{e_i}) = x^{e_0} t_1 x^{e_1} t_2 x^{e_2} \cdots t_r x^{e_r}, \] where $r \geq 0$ and $e_0,e_1,\ldots,e_r \geq 0$.
This rigid word is
\begin{enumerate}[(1)]
\item[$\bullet$] \textit{$n$-limited} if $e_0 + e_1 + \cdots + e_r \leq n$;
\item[$\bullet$] \textit{cube-free} if $e_0,e_1,\ldots,e_r < 3$.
\end{enumerate}
A \textit{rigid identity} is an identity formed by a pair of rigid words: \[ x^{e_0} \prod_{i=1}^r (t_ix^{e_i}) \approx x^{f_0} \prod_{i=1}^r (t_ix^{f_i}); \] this identity is \textit{efficient} if $(e_0,f_0),(e_1,f_1), \ldots,(e_r,f_r) \neq (0,0)$.
If the value of~$r$ is small in the rigid word or rigid identity above, then it is less cumbersome to write distinct variables instead of $t_1,t_2,\ldots,t_r$; for instance, the rigid identity $x^3 t_1 x t_2 x^2 t_3 \approx xt_1t_2x^9t_3x^2$ can be written as $x^3 h x k x^2 t \approx xhkx^9tx^2$.

\begin{lemma}[Lee {\cite[Lemma~8 and Remark~11]{Lee13a}}] \label{L: O subvarieties}
Each noncommutative subvariety of~$\vO$ is defined by the identities~\eqref{id: O} together with finitely many efficient rigid identities.
\end{lemma}

\subsection{Factor monoids}

For any word $\ww \in \sX^*$, the \textit{factor monoid of $\ww$}, denoted by $\mM(\ww)$, is the monoid that consists of all factors of~$\ww$ and a zero element~$0$, with multiplication~$\cdot$ given by \[ \wu \cdot \wv = \begin{cases} \wu\wv & \text{if $\wu\wv$ is a factor of~$\ww$}, \\ 0 & \text{otherwise}; \end{cases} \] the empty word, more conveniently written as~$1$, is the identity element of $\mM(\ww)$.
For example, the monoids~$\mN_6^1$ and~$\mN_8^1$ introduced in Subsection~\ref{subsec: Fu} are isomorphic to the monoids $\mM(xyx)$ and $\mM(xyxy)$, respectively.

A word~$\ww$ is an \textit{isoterm} for a variety~$\vV$ if~$\vV$ violates any nontrivial identity of the form $\ww \approx \wv$.
Equivalently, $\ww$ is an isoterm for~$\vV$ if and only if the identities satisfied by~$\vV$ cannot be used to convert~$\ww$ into a different word.

Given any word~$\ww$, let $\vM(\ww)$ denote the variety generated by the factor monoid $\mM(\ww)$.
One advantage in working with factor monoids is the relative ease of checking if a variety $\vM(\ww)$ is contained in some given variety.

\begin{lemma}[Jackson {\cite[Lemma~3.3]{Jac05}}] \label{L: isoterm}
For any variety~$\vV$ and any word~$\ww$, the inclusion $\vM(\ww) \subseteq \vV$ holds if and only if~$\ww$ is an isoterm for~$\vV$.
\end{lemma}

\section{{\Ocom} varieties} \label{sec: overcom}

For any $n \geq 3$, define the rigid identities
\begin{gather*}
x^n \prod_{i=1}^n t_i \approx \prod_{i=1}^n (t_i x), \tag{$\ida_n$} \\
x^{n-1}t \approx x^{n-2}tx \approx x^{n-3}tx^2 \approx \cdots \approx xtx^{n-2} \approx tx^{n-1}. \tag{$\idb_n$}
\end{gather*}
Let~$\sB_n$ be the set of the~$n$ rigid words that form the identities in~$\idb_n$: \[ \sB_n = \{ x^{n-1}t, \ x^{n-2}tx, \ x^{n-3}tx^2, \, \ldots, \ xtx^{n-2}, \ tx^{n-1} \}. \]

Recall that~$\vO$ is the variety defined by the identities~\eqref{id: O}; let $\vA_n = \vO\{\ida_n\}$.

\begin{theorem} \label{T: O}
For each $n \geq 3$, the lattice $\lEq(\sB_n)$ is anti-isomorphic to the subinterval $[\vA_n\{\idb_n\},\vA_n]$ of $[\vCom,\vO]$.
Consequently, the {\ocom} variety~$\vO$ is {\fu}.
\end{theorem}

The proof of Theorem~\ref{T: O} requires some intermediate results.

\begin{lemma} \label{L: isoterms for An}
Let $n \geq 3$.
\begin{enumerate}[\rm(i)]
\item Every $(n-1)$-limited rigid word is an isoterm for $\vA_n$.
In particular, every word in~$\sB_n$ is an isoterm for~$\vA_n$.
\item Every $(n-2)$-limited rigid word is an isoterm for $\vA_n\{\idb_n\}$.
\end{enumerate}
\end{lemma}

\begin{proof}
(i) It is routinely checked that the factor monoid $\mM(\ww)$ of any $(n-1)$-limited rigid word~$\ww$ satisfies the identities $\{ \eqref{id: O},\ida_n \}$, so that $\vM(\ww) \subseteq \vA_n$.
The result then follows from Lemma~\ref{L: isoterm}.

(ii) It is routinely checked that the factor monoid $\mM(\ww)$ of any $(n-2)$-limited rigid word~$\ww$ satisfies the identities $\{ \eqref{id: O},\ida_n,\idb_n \}$, so that $\mM(\ww) \subseteq \vA_n\{\idb_n\}$.
The result then follows from Lemma~\ref{L: isoterm}.
\end{proof}

\begin{lemma} \label{sec: id of An pi}
Let $\pi \in \lEq(\sB_n)$.
Suppose that the variety $\vA_n\{\Id(\pi)\}$ satisfies some identity $\wu \approx \wv$ with $\wu \in \sB_n$.
Then $\wv \in \sB_n$ and $(\wu,\wv) \in \pi$.
\end{lemma}

\begin{proof}
It suffices to assume that $\wu \neq \wv$.
Since the deduction $\{ \eqref{id: O}, \ida_n, \Id(\pi) \} \vdash \wu \approx \wv$ holds by Theorem~\ref{T: deduction}, there exists a finite sequence $\wu = \ww_0, \ww_1, \ldots, \ww_m = \wv$ of distinct words such that each identity $\ww_i \approx \ww_{i+1}$ is directly deducible from some identity in $\{ \eqref{id: O}, \ida_n, \Id(\pi) \}$.
Then by Lemma~\ref{L: isoterms for An}(i), the word $\ww_0 = \wu \in \sB_n$ is an isoterm for~$\vA_n$, so that $\{ \eqref{id: O}, \ida_n \} \nvdash \ww_0 \approx \ww_1$ by Theorem~\ref{T: deduction}.
It follows that $\ww_0 \approx \ww_1$ is directly deducible from some identity $\ws \approx \wt$ in~$\Id(\pi)$.
Since $\ww_0,\ws,\wt \in \sB_n$, it is easily seen that $\{ \ww_0,\ww_1 \} = \{ \ws,\wt \}$.
Hence $\ww_1 \in \sB_n$ and $(\ww_0,\ww_1) \in \pi$.

The above argument can be repeated so that by induction, $\ww_{i+1} \in \sB_n$ and $(\ww_i,\ww_{i+1}) \in \pi$ for all $i=0,1,\ldots,m-1$.
Therefore $\wv = \ww_m \in \sB_n$ and, since~$\pi$ is an equivalence relation, $(\wu,\wv) = (\ww_0,\ww_m) \in \pi$.
\end{proof}

For the rest of this section, the mapping $\Phi: \lEq(\sB_n) \to [\vA_n\{\idb_n\},\vA_n]$ given by \[\Phi(\pi) = \vA_n\{\Id(\pi)\} \] is shown to be an anti-isomorphism.
The proof of Theorem~\ref{T: O} is thus complete.

\subsection*{The mapping $\Phi$ is injective}

Suppose that $\Phi(\pi) = \Phi(\rho)$ for some $\pi,\rho \in\lEq(\sB_n)$, so that $\vA_n\{\Id(\pi)\} = \vA_n\{\Id(\rho)\}$.
If $(\wu,\wv) \in \rho$, then the variety $\vA_n\{\Id(\pi)\}$ satisfies the identity $\wu \approx \wv$, whence $(\wu,\wv) \in \pi$ by Lemma~\ref{sec: id of An pi}.
Therefore the inclusion $\rho \subseteq \pi$ holds; the reverse inclusion $\rho \supseteq \pi$ holds by a symmetrical argument, thus $\pi = \rho$.

\subsection*{The mapping $\Phi$ is surjective}

It suffices to show that for any variety $\vV \in [\vA_n\{\idb_n\},\vA_n]$, there exists some $\pi \in \lEq(\sB_n)$ such that $\Phi(\pi) = \vV$.
Since $\Phi(\varepsilon_{\sB_n}) = \vA_n \{ \Id(\varepsilon_{\sB_n}) \} = \vA_n$, suppose that $\vV \neq \vA_n$.
Then by Lemma~\ref{L: O subvarieties}, there exists a finite nontrivial set~$\Sigma$ of efficient rigid identities such that $\vV = \vA_n\Sigma$; since $\vV \neq \vA_n$, the identities in~$\Sigma$ can be chosen to be violated by~$\vA_n$.
It is shown below that any identity $\wu \approx \wv$ in~$\Sigma$ is equivalent to some set of identities from~$\idb_n$.
It follows that $\vV = \vA_n\Sigma'$ for some $\Sigma' \subseteq \idb_n$.
By Lemma~\ref{L: equivalence}, there exists some $\pi \in \lEq(\sB_n)$ such that $\vV = \vA_n\{\Id(\pi)\}$, so that $\Phi(\pi) = \vV$ as required.

Since $\wu \approx \wv$ is an efficient rigid identity that is violated by~$\vA_n$, \[ \wu = x^{e_0} \prod_{i=1}^r (t_ix^{e_i}) \quad \text{and} \quad \wv = x^{f_0} \prod_{i=1}^r (t_ix^{f_i}) \] for some $r \geq 0$ and $e_0,f_0, e_1,f_1,\ldots, e_r,f_r \geq 0$.
Let $e = \sum_{i=0}^r e_i$ and $f = \sum_{i=0}^r f_i$.
Then $e=f$ because $\wu \approx \wv$ is satisfied by the variety~$\vCom$.
If $e=f \geq n$, then the identity~$\ida_n$ can be used to convert both~$\wu$ and~$\wv$ into the same word $x^n \prod_{i=1}^r t_i$, whence the variety $\vA_n$ satisfies the identity $\wu \approx \wv$, a contradiction.
Therefore $e=f \leq n-1$.
But since the variety $\vA_n\{\idb_n\}$ satisfies the identity $\wu \approx \wv$, it follows from Lemma~\ref{L: isoterms for An}(ii) that the words~$\wu$ and~$\wv$ cannot be $(n-2)$-limited.
Consequently, $e=f=n-1$.

Now since the variety $\vA_n\{\idb_n\}$ satisfies the nontrivial identity $\wu \approx \wv$, some identity from $\{ \eqref{id: O}, \ida_n,\idb_n \}$ must be able to convert~$\wu$ into a different word.
By Lemma~\ref{L: isoterms for An}(i), the $(n-1)$-limited rigid word~$\wu$ is an isoterm for the variety $\vA_n$, so none of the identities $\{ \eqref{id: O}, \ida_n \}$ can convert~$\wu$ into a different word.
Therefore only some identities from~$\idb_n$ can be used to convert~$\wu$ into a different word; in this case, it is easily seen that at most two of the exponents $e_0,e_1,\ldots,e_r$ can be nonzero.
By a symmetrical argument, at most two of the exponents $f_0,f_1,\ldots,f_r$ can be nonzero.
The efficiency of $\wu \approx \wv$ implies that $r \leq 3$, while $\wu \neq \wv$ implies that $r \geq 1$.
Therefore the only possibilities are $r = 1,2,3$.
In the following, it is less cumbersome to write $h,k,t$ in place of $t_1,t_2,t_3$, respectively.

\medskip

\noindent{\sc Case~1:} $r=1$.
Up to symmetry, the identity $\wu \approx \wv$ is one of the following:
\begin{enumerate}[\quad(i)]
\item $x^{n-1} h \approx h x^{n-1}$, \item $x^p h x^q \approx x^{n-1} h$, \item $x^p h x^q \approx x^{p'} h x^{q'}$,
\end{enumerate}
where $p,q,p',q' \geq 1$ and $p+q = p'+q' = n-1$.
Then $\wu \approx \wv$ is clearly in~$\idb_n$.

\medskip

\noindent{\sc Case~2:} $r=2$. Up to symmetry, the identity $\wu \approx \wv$ is one of the following:
\begin{enumerate}[\quad(i)]
\item[(iv)] $x^p h x^q k \approx x^{p'} hk x^{q'}$, \item[(v)] $x^p h x^q k \approx h x^{p'} k x^{q'}$, \item[(vi)] $x^p h x^q k \approx hk x^{n-1}$,
\end{enumerate}
where $p,q,p',q' \geq 1$ and $p+q = p'+q' = n-1$.
\begin{enumerate}[ \ {2.}1:]
\item $\wu \approx \wv$ is (iv).
Then the following are deducible from (iv): \[ {\rm(a)}: x^p h x^q \approx x^{p'} h x^{q'}, \quad {\rm(b)}: x^{n-1} k \approx x^{p'} k x^{q'}; \] conversely, (iv) is deducible from $\{ \rm(a),(b) \}$ because \[ x^p h x^q k \stackrel{\rm(a)}{\approx} x^{p'} h x^{q'} k \stackrel{\rm(b)}{\approx} x^{n-1} hk \stackrel{\rm(b)}{\approx} x^{p'} hk x^{q'}. \]
Therefore $\rm(iv) \sim \{ (a),(b) \}$.
\item $\wu \approx \wv$ is (v).
Then the following are deducible from (v): \[ {\rm(a)}: x^p h x^q \approx h x^{n-1}, \quad {\rm(b)}: x^{n-1} k \approx x^{p'} k x^{q'}; \] conversely, (v) is deducible from $\{ \rm(a),(b) \}$ because \[ x^p h x^q k \stackrel{\rm(a)}{\approx} h x^{n-1} k \stackrel{\rm(b)}{\approx} h x^{p'} k x^{q'}. \]
Therefore $\rm(v) \sim \{ (a),(b) \}$.
\item $\wu \approx \wv$ is (vi).
Then the following are deducible from (vi): \[ {\rm(a)} : x^p h x^q \approx h x^{n-1}, \quad {\rm(b)} : x^{n-1} k \approx k x^{n-1}; \] conversely, (vi) is deducible from $\{ \rm(a),(b) \}$ because \[ x^p h x^q k \stackrel{\rm(a)}{\approx} h x^{n-1} k \stackrel{\rm(b)}{\approx} hk x^{n-1}. \]
Therefore $\rm(vi) \sim \{ (a),(b) \}$.
\end{enumerate}

\medskip

\noindent{\sc Case~3:} $r=3$. Up to symmetry, the identity $\wu \approx \wv$ is one of the following:
\begin{enumerate}[\quad(i)]
\item[(vii)] $x^p h x^q kt \approx hk x^{p'} t x^{q'}$, \item[(viii)] $x^p hk x^q t \approx h x^{p'} kt x^{q'}$, \item[(ix)] $x^p hkt x^q \approx h x^{p'} k x^{q'} t$,
\end{enumerate}
where $p,q,p',q' \geq 1$ and $p+q = p'+q' = n-1$.
\begin{enumerate}[ \ {3.}1:]
\item $\wu \approx \wv$ is (vii).
Then the following are deducible from (vii): \[ {\rm(a)}: x^p h x^q \approx h x^{n-1}, \quad {\rm(b)}: x^{n-1} k \approx k x^{n-1}, \quad {\rm(c)}: x^{n-1} t \approx x^{p'} t x^{q'}; \] conversely, (vii) is deducible from $\{ \rm(a),(b),(c) \}$ because \[ x^p h x^q kt \stackrel{\rm(a)}{\approx} h x^{n-1} kt \stackrel{\rm(b)}{\approx} hk x^{n-1} t \stackrel{\rm(c)}{\approx} hk x^{p'} t x^{q'}. \]
Therefore $\rm(vii) \sim \{ (a),(b),(c) \}$.
\item $\wu \approx \wv$ is (viii).
Then the following are deducible from (viii): \[ {\rm(a)}: x^p h x^q \approx h x^{n-1}, \quad {\rm(b)}: x^p k x^q \approx x^{p'} k x^{q'}, \quad {\rm(c)}: x^{n-1} t \approx x^{p'} t x^{q'}; \] conversely, (viii) is deducible from $\{ \rm(a),(b),(c) \}$ because \[ x^p hk x^q t \stackrel{\rm(a)}{\approx} hk x^{n-1} t \stackrel{\rm(a)}{\approx} h x^p k x^q t \stackrel{\rm(b)}{\approx} h x^{p'} k x^{q'} t \stackrel{\rm(c)}{\approx} h x^{n-1} kt \stackrel{\rm(c)}{\approx} h x^{p'} kt x^{q'}. \]
Therefore $\rm(viii) \sim \{ (a),(b),(c) \}$.
\item $\wu \approx \wv$ is (ix).
Then the following are deducible from (ix): \[ {\rm(a)}: x^p h x^q \approx h x^{n-1}, \quad {\rm(b)}: x^p k x^q \approx x^{p'} k x^{q'}, \quad {\rm(c)}: x^p t x^q \approx x^{n-1} t; \] conversely, (ix) is deducible from $\{ \rm(a),(b),(c) \}$ because \[ x^p hkt x^q \stackrel{\rm(c)}{\approx} x^{n-1} hkt \stackrel{\rm(c)}{\approx} x^p hk x^q t \stackrel{\rm(a)}{\approx} hk x^{n-1} t \stackrel{\rm(a)}{\approx} h x^p k x^q t \stackrel{\rm(b)}{\approx} h x^{p'} k x^{q'} t. \]
Therefore $\rm(ix) \sim \{ (a),(b),(c) \}$.
\end{enumerate}
In any case, $\wu \approx \wv$ is equivalent to some set of identities from~$\idb_n$.

\subsection*{The mapping $\Phi$ is an anti-isomorphism}

Let $\pi,\rho \in \lEq(\sB_n)$.
If $\pi \subseteq \rho$, then the inclusion $\vA_n\{\Id(\rho)\} \subseteq \vA_n\{\Id(\pi)\}$ holds, so that $\Phi(\rho) \subseteq \Phi(\pi)$.
Conversely, assume the inclusion $\Phi(\rho) \subseteq \Phi(\pi)$, so that $\vA_n\{\Id(\rho)\} \subseteq \vA_n\{\Id(\pi)\}$.
Then for any $(\wu,\wv) \in \pi$, the identity $\wu \approx \wv$ is satisfied by $\vA_n\{\Id(\rho)\}$, whence $(\wu,\wv) \in \rho$ by Lemma~\ref{sec: id of An pi}.
Therefore $\pi \subseteq \rho$.

\section{Aperiodic varieties} \label{sec: periodic}

For any $n \geq 2$, define the rigid identities
\begin{align*}
x^4 \approx x^3, \quad x^3t \approx tx^3, & \quad x^{2n} \prod_{i=1}^{2n} t_i \approx \prod_{i=1}^{2n} (t_i x), \tag{$\idc_n$} \\
xt_1x^2t_2x^2 \cdots t_{n-2} x^2 t_{n-1}x^2 & \approx x^2t_1xt_2x^2 \cdots t_{n-2} x^2 t_{n-1}x^2 \tag{$\idd_n$} \\
& \approx x^2t_1x^2t_2x \cdots t_{n-2} x^2 t_{n-1}x^2 \\
& \,\ \vdots \\
& \approx x^2t_1x^2t_2x^2 \cdots t_{n-2} x t_{n-1}x^2 \\
& \approx x^2t_1x^2t_2x^2 \cdots t_{n-2} x^2 t_{n-1}x.
\end{align*}
Let~$\sD_n$ be the set of the~$n$ rigid words that form the identities in~$\idd_n$:
\[
\sD_n = \left\{
\begin{array}{r}
xt_1x^2t_2x^2 \cdots t_{n-2} x^2 t_{n-1}x^2, \ x^2t_1xt_2x^2 \cdots t_{n-2} x^2 t_{n-1}x^2, \, \ldots \\[0.05in]
\ldots, \ x^2t_1x^2t_2x^2 \cdots t_{n-2} x^2 t_{n-1}x
\end{array}
\right\}.
\]

Recall that~$\vO$ is the variety defined by the identities~\eqref{id: O}; let $\vC_n = \vO\{\idc_n\}$.
For each $n \geq 2$, the variety $\vC_n$ contains only finitely many subvarieties \cite[Theorem~4]{Lee13a} and is a subvariety of~$\vE_3$.

\begin{theorem} \label{T: periodic}
For each $n \geq 2$, the lattice $\lEq(\sD_n)$ is anti-isomorphic to the subinterval $[\vC_n\{\idd_n\},\vC_n]$ of $\lL(\vC_n)$.
Consequently, the aperiodic variety~$\vE_3$ is {\fu}.
\end{theorem}

The proof of Theorem~\ref{T: periodic} requires some intermediate results.

\begin{lemma} \label{L: isoterms for Cn}
Let $n \geq 2$.
\begin{enumerate}[\rm(i)]
\item Every $(2n-1)$-limited cube-free rigid word is an isoterm for $\vC_n$.
In particular, every word in~$\sD_n$ is an isoterm for~$\vC_n$.
\item Every $(2n-2)$-limited cube-free rigid word is an isoterm for $\vC_n\{\idd_n\}$.
\end{enumerate}
\end{lemma}

\begin{proof}
(i) It is routinely checked that the factor monoid $\mM(\ww)$ of any $(2n-1)$-limited cube-free rigid word~$\ww$ satisfies the identities $\{ \eqref{id: O},\idc_n \}$, so that $\vM(\ww) \subseteq \vC_n$.
The result then follows from Lemma~\ref{L: isoterm}.

(ii) It is routinely checked that the factor monoid $\mM(\ww)$ of any $(2n-2)$-limited cube-free rigid word~$\ww$ satisfies the identities $\{ \eqref{id: O},\idc_n,\idd_n \}$, so that $\vM(\ww) \subseteq \vC_n\{\idd_n\}$.
The result then follows from Lemma~\ref{L: isoterm}.
\end{proof}

\begin{lemma} \label{sec: id of Cn pi}
Let $\pi \in \lEq(\sD_n)$ and let $\wu \approx \wv$ be any identity satisfied by the variety $\vC_n\{\Id(\pi)\}$.
\begin{enumerate}[\rm(i)]
\item Suppose that~$\wu$ is a $(2n-1)$-limited cube-free rigid word.
Then~$\wv$ is also a $(2n-1)$-limited cube-free rigid word.
\item Suppose that $\wu \in \sD_n$.
Then $\wv \in \sD_n$ and $(\wu,\wv) \in \pi$.
\end{enumerate}
\end{lemma}

\begin{proof}
It suffices to assume that $\wu \neq \wv$.
Since the deduction $\{ \eqref{id: O}, \idc_n, \Id(\pi) \} \vdash \wu \approx \wv$ holds by Theorem~\ref{T: deduction}, there exists a finite sequence $\wu = \ww_0, \ww_1, \ldots, \ww_m = \wv$ of distinct words such that each identity $\ww_i \approx \ww_{i+1}$ is directly deducible from some identity in $\{ \eqref{id: O}, \idc_n, \Id(\pi) \}$.

(i) Suppose that~$\wu$ is a $(2n-1)$-limited cube-free rigid word.
Then by Lem\-ma~\ref{L: isoterms for Cn}(i), the word $\ww_0 = \wu$ is an isoterm for~$\vC_n$.
Therefore the identity $\ww_0 \approx \ww_1$ can only be directly deducible from some identity in $\Id(\pi)$.
Since any identity in $\Id(\pi)$ is formed by a pair of words from~$\sD_n$, it is easily seen that~$\ww_1$ is a $(2n-1)$-limited cube-free rigid word.

The above argument can be repeated so that for all $i=0,1,\ldots,m-1$, the identity $\ww_i \approx \ww_{i+1}$ can only be directly deducible from some identity in $\Id(\pi)$, and~$\ww_{i+1}$ is a $(2n-1)$-limited cube-free rigid word; in particular, $\wv = \ww_m$ is such a word.

(ii) By part~(i), each identity $\ww_i \approx \ww_{i+1}$ can only be directly deducible from some identity in $\Id(\pi)$.
Therefore if $\ww_0 = \wu \in \sD_n$, then it is easy to see that $\ww_{i+1} \in \sD_n$ and $(\ww_i,\ww_{i+1}) \in \pi$
for all $i=0,1,\ldots,m-1$.
Hence $\wv = \ww_m \in \sD_n$ and, since~$\pi$ is an equivalence relation, $(\wu,\wv) = (\ww_0,\ww_m) \in \pi$.
\end{proof}

For the rest of this section, the mapping $\Lambda: \lEq(\sD_n) \to [\vC_n\{\idd_n\},\vC_n]$ given by \[ \Lambda(\pi) = \vC_n\{\Id(\pi)\} \] is shown to be an anti-isomorphism.
The proof of Theorem~\ref{T: periodic} is thus complete.

\subsection*{The mapping $\Lambda$ is injective}

Suppose that $\Lambda(\pi) = \Lambda(\rho)$ for some $\pi,\rho \in\lEq(\sD_n)$, so that $\vC_n\{\Id(\pi)\} = \vC_n\{\Id(\rho)\}$.
If $(\wu,\wv) \in \rho$, then the variety $\vC_n\{\Id(\pi)\}$ satisfies the identity $\wu \approx \wv$, whence $(\wu,\wv) \in \pi$ by Lemma~\ref{sec: id of Cn pi}.
Therefore the inclusion $\rho \subseteq \pi$ holds; the reverse inclusion $\rho \supseteq \pi$ holds by a symmetrical argument, thus $\pi = \rho$.

\subsection*{The mapping $\Lambda$ is surjective}

It suffices to show that for any variety $\vV \in [\vC_n\{\idd_n\},\vC_n]$, there exists some $\pi \in \lEq(\sD_n)$ such that $\Lambda(\pi) = \vV$.
Since $\Lambda(\varepsilon_{\sD_n}) = \vC_n \{ \Id(\varepsilon_{\sD_n}) \} = \vC_n$, suppose that $\vV \neq \vC_n$.
Then by Lemma~\ref{L: O subvarieties}, there exists a finite nontrivial set~$\Sigma$ of efficient rigid identities such that $\vV = \vC_n\Sigma$; since $\vV \neq \vC_n$, the identities in~$\Sigma$ can be chosen to be violated by~$\vC_n$.
It is shown below that any identity $\wu \approx \wv$ in~$\Sigma$ is from~$\idd_n$.
By Lemma~\ref{L: equivalence}, there exists some $\pi \in \lEq(\sD_n)$ such that $\vV = \vC_n\{\Id(\pi)\}$, so that $\Lambda(\pi) = \vV$ as required.

Since $\wu \approx \wv$ is an efficient rigid identity that is violated by~$\vC_n$, \[ \wu = x^{e_0} \prod_{i=1}^r (t_ix^{e_i}) \quad \text{and} \quad \wv = x^{f_0} \prod_{i=1}^r (t_ix^{f_i}) \] for some $r \geq 0$ and $e_0,f_0, e_1,f_1,\ldots, e_r,f_r \geq 0$.
Let $e = \sum_{i=0}^r e_i$ and $f = \sum_{i=0}^r f_i$.
If either $e \geq 2n$ or~$\wu$ is not cube-free, then it follows from Lemma~\ref{sec: id of Cn pi}(i) that either $f \geq 2n$ or~$\wv$ is not cube-free, whence~$\vC_n$ satisfies the identities $\wu \approx x^3 \prod_{i=1}^r t_i \approx \wv$, contradicting the choice of identities in~$\Sigma$.
Therefore by Lemma~\ref{sec: id of Cn pi}(i),
\begin{enumerate}[(a)]
\item $e,f \leq 2n-1$ and
\item both~$\wu$ and~$\wv$ are cube-free.
\end{enumerate}
Since the identity $\wu \approx \wv$ is nontrivial and is satisfied by the variety $\vC_n\{\idd_n\}$, the words~$\wu$ and~$\wv$ cannot be isoterms for $\vC_n\{\idd_n\}$.
Therefore by~(a) and Lemma~\ref{L: isoterms for Cn}(ii),
\begin{enumerate}[(a)]
\item [(c)] $e = f = 2n-1$.
\end{enumerate}

Suppose that $e_j = 0$ for some~$j$.
Then since $\wu \approx \wv$ is efficient and~(b) holds, $f_j \in \{ 1,2 \}$.
But it is easily seen that the identities $\{ \eqref{id: O}, \idc_n, \idd_n \}$ can only convert~$\wv$ into a cube-free rigid word $x^{d_0} \prod_{i=1}^r (t_ix^{d_i})$ with $\sum_{i=1}^r d_i = 2n-1$ and $d_j \neq 0$; in particular, the identities $\{ \eqref{id: O}, \idc_n, \idd_n \}$ cannot convert~$\wu$ into~$\wv$.
This implies that the variety $\vC_n\{\idd_n\}$ violates the identity $\wu \approx \wv$, which is impossible.
A similar contradiction is deduced if $f_j=0$ for some~$j$.
Therefore by~(b),
\begin{enumerate}[(a)]
\item[(d)] $e_0,f_0, e_1,f_1,\ldots, e_r,f_r \in \{1,2\}$.
\end{enumerate}

Now since the variety $\vC_n\{\idd_n\}$ satisfies the nontrivial identity $\wu \approx \wv$, some identity from $\{ \eqref{id: O}, \idc_n,\idd_n \}$ must be able to convert~$\wu$ into a different word.
By~(c), (d), and Lemma~\ref{L: isoterms for Cn}(i), the rigid word~$\wu$ is an isoterm for the variety $\vC_n$, so none of the identities $\{ \eqref{id: O}, \idc_n \}$ can convert~$\wu$ into a different word.
Therefore only some identities from~$\idd_n$ can be used to convert~$\wu$ into a different word; in this case, in view of~(c) and~(d), it is easily seen that $\wu \in \sD_n$.
Then $\wv \in \sD_n$ by Lemma~\ref{sec: id of Cn pi}(ii).
Consequently, $\wu \approx \wv \in \idd_n$.

\subsection*{The mapping $\Lambda$ is an anti-isomorphism}

Let $\pi,\rho \in \lEq(\sD_n)$.
If $\pi \subseteq \rho$, then the inclusion $\vC_n\{\Id(\rho)\} \subseteq \vC_n\{\Id(\pi)\}$ holds, so that $\Lambda(\rho) \subseteq \Lambda(\pi)$.
Conversely, assume the inclusion $\Lambda(\rho) \subseteq \Lambda(\pi)$, so that $\vC_n\{\Id(\rho)\} \subseteq \vC_n\{\Id(\pi)\}$.
Then for any $(\wu,\wv) \in \pi$, the identity $\wu \approx \wv$ is satisfied by $\vC_n\{\Id(\rho)\}$, whence $(\wu,\wv) \in \rho$ by Lemma~\ref{sec: id of Cn pi}(ii).
Therefore $\pi \subseteq \rho$.

\section{Finitely generated varieties} \label{sec: FG}

Recall from Subsection~\ref{subsec: main results} that~$\vE_m$ is the subvariety of~$\vO$ defined by the identities \[ x^{m+1} \approx x^m, \quad x^mt \approx tx^m; \tag{$\ide_m$} \] in other words, $\vE_m = \vO\{\ide_m\}$.
The variety $\vE_2$ is not {\fu} because the lattice $\lL(\vE_2)$ coincides with the chain \[ \mathbf{0} \subset \vE_1 \subset \vM(x) \subset \vM(xt) \subset \vM(xt_1x) \subset \vM(xt_1xt_2x) \subset \cdots \subset \vE_2, \] where~$\mathbf{0}$ is the variety of trivial monoids \cite[Proposition~4.1]{Lee15}.
But by Theorem~\ref{T: periodic}, the variety~$\vE_m$ is {\fu} for all $m \geq 3$.

The following result demonstrates the existence of {\fu} varieties that are finitely generated.

\begin{theorem} \label{T: FG}
For each $m \geq 2$, the variety~$\vE_m$ is contained in some finitely generated variety.
\end{theorem}

In Subsection~\ref{subsec: decomposition}, the variety~$\vE_m$ is decomposed into the complete join of some of its subvarieties.
Then in Subsection~\ref{subsec: containing Em}, all these subvarieties of~$\vE_m$ are shown to be contained in a certain finitely generated variety~$\vT_m^1$.
Consequently, $\vE_m$ is a subvariety of~$\vT_m^1$, whence Theorem~\ref{T: FG} is established.

\subsection{A join decomposition of~$\vE_m$} \label{subsec: decomposition}

Recall from Subsection~\ref{subsec: rigid} that a rigid word \[ x^{e_0} \prod_{i=1}^r (t_ix^{e_i}) = x^{e_0} t_1 x^{e_1} t_2 x^{e_2} \cdots t_r x^{e_r} \] is \textit{cube-free} if $e_0,e_1,\ldots,e_r < 3$.
More generally, for $m \geq 2$, this rigid word is \textit{$m$-free} if $e_0,e_1,\ldots,e_r < m$.
For each $m \geq 2$ and $r \geq 1$, define the $m$-free rigid word \[ \ww_{m,r} = x^{m-1} \prod_{i=1}^r (t_ix^{m-1}) = x^{m-1} t_1 x^{m-1} t_2 x^{m-1} \cdots t_r x^{m-1}. \]

\begin{lemma} \label{L: Em generators}
$\vE_m = \bigvee_{r\geq1} \vM(\ww_{m,r})$ for all $m \geq 2$.
\end{lemma}

\begin{proof}
Let~$\vJ_m$ denote the complete join $\bigvee_{r\geq1} \vM(\ww_{m,r})$.
Then by Lemma~\ref{L: isoterm}, the $m$-free rigid words $\ww_{m,1}, \ww_{m,2}, \ww_{m,3},\ldots$ are isoterms for~$\vJ_m$.
In fact, it is routinely shown that every $m$-free rigid word is an isoterm for~$\vJ_m$.

It is easy to check that $\mM(\ww_{m,r}) \in \vE_m$ for all $r \geq 1$, thus $\vJ_m \subseteq \vE_m$.
Seeking a contradiction, suppose that $\vJ_m \neq \vE_m$.
The variety~$\vJ_m$ is non\-com\-mu\-ta\-tive because it contains the noncommutative monoid $\mM(\ww_{m,1})$.
Therefore it follows from Lemma~\ref{L: O subvarieties} that $\vJ_m$ satisfies some rigid identity $\wu \approx \wv$ that is violated by $\vE_m$; in particular, $\wu \neq \wv$.
Since all $m$-free rigid words are isoterms for~$\vJ_m$, by Lemma~\ref{L: isoterm}, the rigid words~$\wu$ and~$\wv$ cannot be $m$-free.
The deduction $\ide_m \vdash \wu \approx \wv$ is then easily established, whence~$\vE_m$ contradictorily satisfies $\wu \approx \wv$.
\end{proof}

\subsection{Finitely generated variety containing $\vE_m$} \label{subsec: containing Em}

For $m \geq 2$, a semigroup is \textit{$m$-testable} if it satisfies any identity $\wu \approx \wv$ such that
\begin{enumerate}[ \ (a)]
\item the prefix of~$\wu$ of length $m-1$ equals the prefix of~$\wv$ of length $m-1$,
\item the suffix of~$\wu$ of length $m-1$ equals the suffix of~$\wv$ of length $m-1$, and
\item the set of factors of~$\wu$ of length~$m$ equals the set of factors of~$\wv$ of length~$m$.
\end{enumerate}
The class of $m$-testable semigroups forms a variety that is generated by some finite semigroup~$\mT_m$, and~$\mT_m$ satisfies an identity $\wu \approx \wv$ if and only if the conditions in (a)--(c) hold; see Trahtman~\cite{Tra99} for more information.

Let~$\vT_m^1$ denote the variety of monoids generated by~$\mT_m^1$.
Since the semigroup~$\mT_m$ satisfies the identity $x^{m+1} \approx x^m$ \cite{Tra99}, the variety~$\vT_m^1$ is aperiodic.

\begin{proposition} \label{P: Tn}
$\vE_m \subseteq \vT_m^1$ for all $m \geq 2$.
\end{proposition}

\begin{proof}
In view of Lemma~\ref{L: Em generators}, it suffices to show that $\vM(\ww_{m,r}) \subseteq \vT_m^1$ for all $r \geq 1$.
Suppose that~$\vT_m^1$ satisfies some nontrivial identity $\wu \approx \wv$, where $\wu = \ww_{m,r}$ for some $r \geq 1$.
The word~$xy$ is clearly an isoterm for~$\vT_m^1$, so that $\vM(xy) \subseteq \vT_m^1$ by Lemma~\ref{L: isoterm}.
It is then easily shown that $\wv = x^{e_0} \prod_{i=1}^r (t_ix^{e_i})$ for some $e_0,e_1, \ldots, e_r \geq 0$; see, for example, Gusev and Vernikov \cite[Proposition~2.13]{GV18}.
Since the semigroup~$T_m$ satisfies $\wu \approx \wv$, so that the conditions in (a)--(c) hold for this identity, it follows that $\wv = \ww_{m,r} = \wu$.
Therefore for all $r \geq 1$, the word $\ww_{m,r}$ is an isoterm for~$\vT_m^1$, whence $\vM(\ww_{m,r}) \subseteq \vT_m^1$ by Lemma~\ref{L: isoterm}.
\end{proof}

\section{Some open problems} \label{sec: problems}

\subsection{Varieties of index~2}
Recall that a variety is periodic if it satisfies the identity $x^{m+k} \approx x^m$ for some $m,k \geq 1$; in this case, the number~$m$ is the \textit{index} of the variety.
Varieties of index~1 are completely regular and so are not {\fu}~\cite{PR90}.
For each $m \geq 3$, the {\fu} variety~$\vT_m^1$ is of index~$m$.
As for varieties of index~2, a {\fu} example have not yet been found.

\begin{question}
Is there a {\fu} variety of monoids of index~2?
Specifically, is the variety of monoids defined by the identity $x^3 \approx x^2$ {\fu}?
\end{question}

It is also of interest to locate a {\fu} variety of index~2 that is finitely generated.
An obvious variety for consideration is~$\vT_2^1$, which is known to be generated by the monoid~$\mA_2^1$ obtained from the 0-simple semigroup \[ \mA_2 = \langle a,b \,|\, a^2=aba=a, \, b^2=0, \, bab=b \rangle \] of order five.
Another possible example is the variety~$\vB_2^1$ generated by the monoid $\mB_2^1$, which is well known to be a proper subvariety of~$\vT_2^1$; see, for example, Lee~\cite{Lee12b}.
The plausibility for~$\vB_2^1$---and so also~$\vT_2^1$---to be {\fu} follows from the complex structure of the lattice $\lL(\vB_2^1)$: it is uncountable and has infinite width \cite{JL18,Lee12b}.

\begin{question}
Which, if any, of the varieties~$\vB_2^1$ and~$\vT_2^1$ is {\fu}?
\end{question}

\subsection{Finitely generated finitely based varieties}

A variety is \textit{finitely based} if it can be defined by a finite set of identities; otherwise, it is \textit{non-finitely based}.
A finitely generated variety is \textit{inherently non-finitely based} if every locally finite variety containing it is non-finitely based.

The variety~$\vB_2^1$ is an example of inherently non-finitely based variety~\cite{JL18}.
Since the inclusions $\vB_2^1 \subseteq \vT_2^1 \subseteq \vT_3^1 \subseteq \vT_4^1 \subseteq \cdots$ hold, the varieties in this chain are all inherently non-finitely based.
In other words, all finitely generated {\fu} varieties exhibited in the present article are non-finitely based; finding a finitely based example is thus of fundamental importance.

\begin{question}
Is there a {\fu} variety of monoids that is both finitely generated and finitely based?
\end{question}

\subsection{Joins of varieties}

The join of two finitely generated varieties is clearly finitely generated, but in general, the join of two varieties with some finiteness property can result in a variety that violates the property.
For instance, there exist varieties of monoids~$\vV_1$ and~$\vV_2$ such that
\begin{enumerate}[(1)]
\item[$\bullet$] $\vV_1$ and~$\vV_2$ are finitely based but their join $\vV_1 \vee \vV_2$ is non-finitely based~\cite{Jac05,JS00,Lee13b,Sap15};
\item[$\bullet$] $\vV_1$ and~$\vV_2$ are small but $\lL(\vV_1 \vee \vV_2)$ is an uncountable lattice that violates both the ascending chain and descending chain conditions~\cite{JL18,Gus19}.
\end{enumerate}
Recall that a variety is \textit{small} if it contains only finitely many subvarieties.

It is natural to question if a {\fu} variety of monoids can be the join of two simpler varieties.

\begin{question} \label{Q: join}
\begin{enumerate}[\rm(i)]
\item Are there varieties of monoids~$\vV_1$ and $\vV_2$ that are not {\fu} such that the join $\vV_1 \vee \vV_2$ is {\fu}?
\item Are there small varieties of monoids~$\vV_1$ and~$\vV_2$ such that the join $\vV_1 \vee \vV_2$ is {\fu}?
\end{enumerate}
\end{question}

Question~\ref{Q: join}(i) has an affirmative answer within the context of varieties of semigroups.
For instance, consider the semigroup \[ \mS_0 = \langle a,b \,|\, a^2 = a^3 = ab, \, ba=b \rangle \] of order three and its sub\-semi\-groups $\mS_1 = \{ a,a^2 \}$ and $\mS_2 = \{ a^2,b \}$.
For each $i \in \{ 0,1,2\}$, let~$\svV_i$ denote the variety of semigroups generated by the monoid~$\mS_i^1$.
Then the inclusion $\svV_1 \vee \svV_2 \subseteq \svV_0$ holds.
But the semigroup~$\mS_0$ is embeddable in $\mS_1 \times \mS_2$, so that $\svV_1 \vee \svV_2 = \svV_0$.
Now the variety~$\svV_0$ is {\fu} while the varieties~$\svV_1$ and~$\svV_2$ are not~\cite{Lee07}.

However, the variety~$\svV_1$ is not small, so that~$\svV_1$ and~$\svV_2$ do not provide an affirmative answer to Question~\ref{Q: join}(ii) within the context of varieties of semigroups.

\subsection{{\Lu} varieties}

It follows from Pudl\'ak and \Tuma~\cite{PT80} that a variety~$\vV$ is {\fu} if and only if for all sufficiently large $n \geq 1$, the lattice $\lEq(n)$ of equivalence relations on $\{ 1,2,\ldots,n\}$ is anti-isomorphic to some sublattice of $\lL(\vV)$.
In the present article, {\fu} varieties~$\vV$ of monoids are exhibited with the stronger property that for all sufficiently large $n \geq 1$, the lattice $\lEq(n)$ is anti-isomorphic to some subinterval of $\lL(\vV)$.

A yet even stronger property that a variety~$\vV$ can satisfy is when the lattice $\lEq(\infty)$ of equivalence relations on $\{ 1,2,3,\ldots\}$ is anti-isomorphic to some subinterval of $\lL(\vV)$; following Shevrin et al.\@~\cite{SVV09}, such a variety is said to be \textit{lattice universal}.
Lattice universal varieties of semigroups have been found by Burris and Nelson~\cite{BN71a} and Je\v{z}ek~\cite{Jez76}; it is natural to question if a variety of monoids can also satisfy this property.

\begin{question}
Is there a variety of monoids that is lattice universal?
\end{question}

Since every subvariety of~$\vO$ is finitely based~\cite{Lee12a}, the variety~$\vO$ contains only countably many subvarieties.
Therefore all subvarieties of~$\vO$, which include~$\vE_m$, cannot be lattice universal.

It is also of interest to know if there exists a variety~$\vV$ of monoids with the weaker property that the lattice $\lEq(\infty)$ is anti-isomorphic to some sublattice of $\lL(\vV)$ and not to some subinterval of $\lL(\vV)$.

\end{document}